\documentclass[final]{amsart}
\usepackage{amsmath}
\usepackage{amssymb}
\usepackage{amsmath, amssymb, amsfonts, amsxtra}
\theoremstyle{proclaim}
\newtheorem{question}[]{Question}
\newtheorem{theorem}{Theorem}[section]
\newtheorem{lemma}[theorem]{Lemma}
\newtheorem{corollary}[theorem]{Corollary}
\newtheorem{proposition}[theorem]{Proposition}

\theoremstyle{fancyproclaim}

\theoremstyle{statement}
\newtheorem{cor}[theorem]{Corollary}
\newtheorem{remark}[theorem]{Remark}

\newtheorem{example}[theorem]{Example}
\theoremstyle{fancystatement}

\numberwithin{equation}{section}

\providecommand{\AMS}{$\mathcal{A}$\kern-.1667em%
	\lower.25em\hbox{$\mathcal{M}$}\kern-.125em$\mathcal{S}$}
\begin{document}
	\title[Characterization $k-$smoothness of  operators ]{Characterization of $k-$smoothness of  operators defined between infinite-dimensional spaces}
\author[Arpita Mal, Subhrajit Dey and Kallol Paul]{Arpita Mal, Subhrajit Dey and Kallol Paul}

\newcommand{\acr}{\newline\indent}

	\address[Mal]{Department of Mathematics\\ Jadavpur University\\ Kolkata 700032\\ West Bengal\\ INDIA}
	\email{arpitamalju@gmail.com}

		\address[Dey]{Department of Mathematics\\ Muralidhar Girls' College\\ Kolkata 700029\\ West Bengal\\ INDIA}
	\email{subhrajitdeyjumath@gmail.com}

		\address[Paul]{Department of Mathematics\\ Jadavpur University\\ Kolkata 700032\\ West Bengal\\ INDIA}
	\email{kalloldada@gmail.com}

	\thanks{The research of  Arpita Mal is supported by UGC, Govt. of India.  The research of Prof. Paul  is supported by project MATRICS(MTR/2017/000059)  of DST, Govt. of India. } 
	

\maketitle
\begin{abstract}
 We characterize $k-$smoothness of bounded linear operators defined between infinite-dimensional  Hilbert spaces. We  study the problem in the setting of both finite and infinite-dimensional Banach spaces.  We also  characterize $k-$smoothness of operators on some particular spaces, namely $\mathbb{L}(\mathbb{X},\ell_{\infty}^n),~\mathbb{L}(\ell_{\infty}^3,\mathbb{Y}),$ where $\mathbb{X}$ is a finite-dimensional Banach space and $\mathbb{Y}$ is a two-dimensional Banach space. As an application, we characterize extreme contractions on $\mathbb{L}(\ell_{\infty}^3,\mathbb{Y}),$ where $\mathbb{Y}$ is a two-dimensional polygonal Banach space.
\end{abstract}

\section{Introduction}

The problem of characterizing $k-$smooth operators defined between arbitrary Banach or Hilbert spaces is relatively  new but an important area of  research  in the field of geometry of Banach spaces. There are several papers including \cite{H,Ha,KS,LR,MP,Wa} that contain the study of $k-$smoothness of operators on different spaces.   In this paper, our objective is to study the $k-$smoothness of  bounded linear operators defined between infinite-dimensional   spaces. We first fix the notations and terminologies to be used throughout the paper.

Let $\mathbb{X},$ $\mathbb{Y}$ denote Banach spaces and $\mathbb{H}$ denote Hilbert space. Throughout the paper we assume that the  spaces are real unless otherwise mentioned. The unit ball and the unit sphere of $\mathbb{X}$ are denoted by $B_\mathbb{X}$ and $S_\mathbb{X}$ respectively, i.e., $B_{\mathbb{X}}=\{x\in \mathbb{X}:\|x\|\leq 1\},S_\mathbb{X}=\{x\in \mathbb{X}:\|x\|= 1\}.$ The space of bounded (compact) linear operators between  $\mathbb{X}$ and  $\mathbb{Y}$ is denoted by  $\mathbb{L}(\mathbb{X},\mathbb{Y})~(\mathbb{K}(\mathbb{X},\mathbb{Y})).$ If $\mathbb{X}=\mathbb{Y},$ then we write  $\mathbb{L}(\mathbb{X},\mathbb{Y}):=\mathbb{L}(\mathbb{X})$ and $\mathbb{K}(\mathbb{X},\mathbb{Y}):=\mathbb{K}(\mathbb{X}).$ $\mathbb{X}^*$ denote the dual space of $\mathbb{X}.$ An element $x \in S_{\mathbb{X}}$ is said to be an extreme point of the convex set $ B_{\mathbb{X}}$ if and only if $ x = (1-t)y + tz$ for some $ y,z \in B_{\mathbb{X}}$ and $ t \in (0,1) $ implies that $ y =z=x.$ The set of all extreme points of  $ B_{\mathbb{X}}$ is denoted by $Ext(B_\mathbb{X}).$   For $x,y\in \mathbb{X},$ let $L[x,y]=\{tx+(1-t)y:0\leq t\leq 1\}$ and $L(x,y)=\{tx+(1-t)y:0< t< 1\}.$ A Banach space $\mathbb{X}$ is said to be a strictly convex Banach space if every element of the unit sphere $ S_{\mathbb{X}}$ is an extreme point of  the unit ball $ B_{\mathbb{X}},$ equivalently, $\mathbb{X}$ is said to be a strictly convex Banach space, if the unit sphere of $\mathbb{X}$ does not contain non-trivial straight line segment.  A face $E$ of a convex set $C$ is said to be an edge if for each $z\in E,$ there exist extreme points  $x,y$ in $C$ such that $z\in L[x,y].$     An element $x^*\in S_{\mathbb{X}^*}$ is  said to be a supporting linear functional of $x\in S_{\mathbb{X}},$ if $x^*(x)=1.$ For a unit vector $x,$ let $J(x)$ denote the set of all supporting linear functionals of $x,$ i.e., $J(x)=\{x^*\in S_{\mathbb{X}^*}:x^*(x)=1\}.$ By the Hahn-Banach Theorem, $J(x)\neq \emptyset,$ for all $x\in S_\mathbb{X}.$ We would like to note that $J(x)$ is a weak*-compact convex subset of $S_{\mathbb{X}^*}.$ The set of all extreme points of $J(x)$ is denoted by $Ext~J(x),$ where $x\in S_\mathbb{X}.$ A unit vector $x$ is said to be a smooth point if $J(x)$ is singleton. $\mathbb{X}$ is said to be a smooth Banach space if every unit vector of $\mathbb{X}$ is smooth. 

In 2005, Khalil and Saleh \cite{KS} generalized the notion of smoothness and introduced the notion of multi-smoothness or $k-$smoothness depending on the ``size" of $J(x).$ An element $x\in S_{\mathbb{X}}$ is said to be $k-$smooth or the order of smoothness of $x$ is $k,$ if $J(x)$ contains exactly $k$ linearly independent supporting linear functionals of $x$. In other words, $x$ is $k-$smooth, if $\dim~span~J(x)=k.$ Moreover, from \cite[Prop. 2.1]{LR}, we get that $x$ is $k-$smooth, if $k=\dim~span~Ext~J(x).$   Similarly, $T\in \mathbb{L}(\mathbb{X},\mathbb{Y})$ is said to be $k-$smooth operator, if $k=\dim~span~J(T)=\dim~span~Ext~J(T).$ Observe that, $1-$smooth points of $S_\mathbb{X}$ are actually the smooth points of $S_\mathbb{X}.$ In our study, the norm attainment set of an operator plays an important role which will be clear in due time. The norm attainment set of $T,$ denoted as $ M_T$,  is defined as the collection of all unit vectors $x$ at which $T$ attains its norm, i.e., $M_T=\{x\in S_\mathbb{X}:\|Tx\|=\|T\|\}.$  The notion of $k-$smoothness  has a nice connection with extreme contraction which will be explored later. An operator $T\in \mathbb{L}(\mathbb{X},\mathbb{Y})$ is said to be an extreme contraction, if $T$ is an extreme point of the unit ball of $\mathbb{L}(\mathbb{X},\mathbb{Y}).$
A two-dimensional Banach space $\mathbb{X}$ is said to be a polygonal Banach space, if $B_\mathbb{X}$ contains only finitely many extreme points. Equivalently, a two-dimensional Banach space $\mathbb{X}$ is a polygonal Banach space, if $B_\mathbb{X}$ is a polygon.

From \cite[Th. 3.8]{LR}, we know that there is a large class of Banach spaces which does not contain $k-$smooth point, where $k\in \mathbb{N}.$ The papers  \cite{H,Ha,KS,LR,MP,Wa} contain extensive study on multi-smoothness in Banach space and in operator space. In \cite[Th. 2.4]{Wa} W\'ojcik studied $k-$smoothness of compact operators defined between complex (real) Hilbert spaces. In this paper, we obtain a complete characterization of $k-$smoothness of bounded linear operators defined between  complex (real) Hilbert spaces. We prove that a bounded linear operator $T$ defined on a complex (real)  Hilbert space $\mathbb{H}$ is  $n^2-$smooth (${{n+1}\choose {2}}-$smooth) if  and only if $ M_T = S_{H_0} ,$ where $ \dim(H_0)=n $ and $\|T\|_{H_0^\perp} < \|T\|.$ Moving onto Banach spaces, the complete characterization of $k-$smooth operators defined between arbitrary Banach spaces is still  not known, in fact it is  elusive even for finite-dimensional Banach spaces. In \cite{MP}, the authors characterized the $k-$smoothness of a bounded linear operator defined between two-dimensional Banach spaces. In this paper, we continue our study in this direction and obtain sufficient conditions for $k-$smoothness of bounded linear operators defined between infinite-dimensional Banach spaces. We also obtain a relation between the order of smoothness of the operators $T$ and $T^*,$ where $T$ is defined between finite-dimensional Banach spaces and $T^*$ is the adjoint of $T$. Using this relation, we characterize the order of smoothness of an operator defined from a finite-dimensional Banach space to $\ell_{\infty}^n,~(n\in \mathbb{N}).$    We also obtain a characterization of the order of smoothness of $T\in \mathbb{L}(\ell_\infty^3,\mathbb{Y}),$ where $\mathbb{Y}$ is a two-dimensional Banach space. As an application of this result, we characterize the extreme contractions on the space $\mathbb{L}(\ell_\infty^3,\mathbb{Y}),$ where $\mathbb{Y}$ is a two-dimensional polygonal Banach space.\\

We state the following lemma \cite[Lemma 3.1]{W}, characterizing $Ext~J(T),$ which will be used often. 
\begin{lemma}\cite[Lemma 3.1]{W}\label{lemma-wojcik}
	Suppose that $\mathbb{X}$ is a reflexive Banach space. Suppose that $\mathbb{K}(\mathbb{X},\mathbb{Y})$ is an $M-$ideal in $\mathbb{L}(\mathbb{X},\mathbb{Y}).$ Let $T\in \mathbb{L}(\mathbb{X},\mathbb{Y}), \|T\|=1$ and  dist$(T,\mathbb{K}(\mathbb{X},\mathbb{Y}))<1.$  Then $M_T\cap Ext(B_\mathbb{X})\neq \emptyset$ and 
\[Ext ~J(T)=\{y^*\otimes x\in \mathbb{K}(\mathbb{X},\mathbb{Y})^*:x\in M_T\cap Ext(B_{\mathbb{X}}), y^*\in Ext ~J(Tx)\},\]
where  $y^*\otimes x: \mathbb{K}(\mathbb{X},\mathbb{Y})\to \mathbb{R}$ is defined by $y^*\otimes x(S)=y^*(Sx)$ for every $S\in \mathbb{K}(\mathbb{X},\mathbb{Y}).$
\end{lemma}
We end this section with the following definition:
A subspace $M$ of a Banach space $\mathbb{X}$ is said to be an $M-$ideal if there exists a projection $P$ on $\mathbb{X}^*$ such that $P(\mathbb{X}^*)=\{x^*\in \mathbb{X}^*:x^*(m)=0~\forall~ m\in M \}$ and for all $x^*\in \mathbb{X}^*,$
$$\|x^*\|=\|P(x^*)\|+\|x^*-P(x^*)\|.$$ 
It is well known that for a Hilbert space $\mathbb{H},$ $\mathbb{K}(\mathbb{H})$ is an $M-$ideal in $\mathbb{L}(\mathbb{H})$ and for each $1<p<\infty,$ $\mathbb{K}(\ell_p)$ is an $M-$ideal in $\mathbb{L}(\ell_p).$ Interested readers are referred to \cite{HWW} for more  information in this topic.

\section{$k-$smoothness of operators defined on Hilbert spaces}

We begin this section with the study of $k-$smooth operators defined between arbitrary Hilbert spaces. We use the notion of Birkhoff-James orthogonality to prove the theorem. Recall that, for $x,y\in \mathbb{X},$ $x$ is said to be Birkhoff-James orthogonal \cite{B,J} to $y,$ written as $x\perp_B y,$ if $\|x+\lambda y\|\geq \|x\|$ for each scalar $\lambda.$ Similarly, for $T,A\in \mathbb{L}(\mathbb{X},\mathbb{Y}),$ we say that $T\perp_B A$ if $\|T+\lambda A\|\geq \|T\|$ for each scalar $\lambda.$

\begin{theorem}\label{th-hil}
	Let $\mathbb{H}_1,\mathbb{H}_2$ be   Hilbert spaces. Let $T\in S_{\mathbb{L}(\mathbb{H}_1,\mathbb{H}_2)}$ be such that $M_T=S_{H_0},$ where $H_0$ is a finite-dimensional subspace of $\mathbb{H}_1$ with $\dim(H_0)=n$ and $\|T\|_{H_0^\perp}<1.$ Then $T$ is $n^2-$smooth, if $\mathbb{H}_1,\mathbb{H}_2$ are considered to be complex and it is ${{n+1}\choose{2}} -$smooth, if   the spaces are considered to be real.
\end{theorem}
\begin{proof}
 We claim that $dist(T,\mathbb{K}(\mathbb{H}_1,\mathbb{H}_2))<1.$ If possible, suppose that this is not true. Then for every $S\in \mathbb{K}(\mathbb{H}_1,\mathbb{H}_2),~dist(T,\text{Span}\{S\})\geq dist (T,\mathbb{K}(\mathbb{H}_1,\mathbb{H}_2))\geq 1,$ i.e., for every scalar $\lambda,\|T-\lambda S\|\geq dist(T,\text{Span}\{S\})\geq 1.$ Therefore, $ T\perp_B S.$ Define $S:\mathbb{H}\to\mathbb{H}$ by $Sx=Tx$ whenever $x\in H_0$ and $Sx=0,$ whenever $x\in H_0^\perp.$ Then clearly, $S\in \mathbb{K}(\mathbb{H}_1,\mathbb{H}_2),$ since $H_0$ is finite-dimensional. Hence, $T\perp_B S.$ By \cite[Th. 3.1]{PSG}, there exists $x\in M_T=S_{H_0}$ such that $Tx\perp Sx,$ i.e., $ Tx\perp Tx,$ a contradiction. This establishes our claim. We note that $\mathbb{K}(\mathbb{H}_1,\mathbb{H}_2)$ is an $M-$ideal in $\mathbb{L}(\mathbb{H}_1,\mathbb{H}_2)$ and  $dist(T,\mathbb{K}(\mathbb{H}_1,\mathbb{H}_2))<1,$ so by Lemma \ref{lemma-wojcik}, $Ext(J(T))=\{y^*\otimes x:x\in M_T,y^*\in Ext (J(Tx))\},$ where $y^*\otimes x (S)=y^*(Sx)=\langle Sx,Tx\rangle$ for every $S\in \mathbb{L}(\mathbb{H}_1,\mathbb{H}_2).$ So we can write $Ext(J(T))=\{x\otimes Tx:x\in M_T\},$ where $x\otimes Tx(S)=\langle Sx,Tx\rangle$ for every $S\in \mathbb{L}(\mathbb{H}_1,\mathbb{H}_2).$  
 
 Suppose that $\{e_1,e_2,\ldots,e_n\}$ is an orthonormal basis of $H_0.$  Assume that $\mathbb{H}_1,\mathbb{H}_2$ are complex Hilbert spaces. Then 
\begin{eqnarray*}
	&&\dim~span~ Ext(J(T))\\
	&=&\dim~span~  \{x\otimes Tx:x\in M_T\}\\
	&=&\dim~span~ \Big\{\sum_{i,j=1}^na_i\overline{a_j}e_i\otimes Te_j:\sum_{i=1}^n|a_i|^2=1\Big\}\\
	&=& \dim~span~  \{e_i\otimes Te_j:1\leq i,j\leq n\}\\
	&=& n^2.
\end{eqnarray*}
Next assume that the Hilbert spaces are real, then  
\begin{eqnarray*}
	&&\dim~span~ Ext(J(T))\\
	&=&\dim~span~ \Big\{\sum_{i,j=1}^na_i{a_j}e_i\otimes Te_j:\sum_{i=1}^n|a_i|^2=1\Big\}\\
	&=& {{n+1}\choose{2}}.
\end{eqnarray*}
Thus, $T$ is $n^2-$smooth, if $\mathbb{H}_1,\mathbb{H}_2$ are considered to be complex and it is ${{n+1}\choose{2}} -$smooth, if   the spaces are considered to be real.
\end{proof}

 Noting that  for a compact operator $T,$ 
 $dist(T,\mathbb{K}(\mathbb{H}_1,\mathbb{H}_2))=0 <1,$ we get the following corollary which provides a sufficient condition  for the $k-$smoothness of a compact operator \cite[Th. 2.4]{Wa} defined between Hilbert spaces.
\begin{cor}
	Let $\mathbb{H}_1,\mathbb{H}_2$ be   Hilbert spaces. Let $T\in S_{\mathbb{K}(\mathbb{H}_1,\mathbb{H}_2)}$ be such that $M_T=S_{H_0},$ where $H_0$ is a finite-dimensional subspace of $\mathbb{H}_1$ with $\dim(H_0)=n.$ Then $T$ is $n^2-$smooth, if $\mathbb{H}_1,\mathbb{H}_2$ are considered to complex and it is ${{n+1}\choose{2}} -$smooth, if   the spaces are considered to be real.
	\end{cor}

Next we show that the conditions mentioned in Theorem \ref{th-hil} are necessary for $k-$smoothness.
To do so, we need  the following theorem on Birkhoff-James orthogonality for complex Hilbert spaces, in case of real Hilbert spaces analogous theorem can be obtained from \cite[Th. 3.2]{MPRS}. 

\begin{theorem}\label{th-complexhil}
	Let $\mathbb{H}_1,\mathbb{H}_2$ be complex Hilbert spaces. Let $T\in S_{\mathbb{L}(\mathbb{H}_1,\mathbb{H}_2)}$ be such that $dist(T,\mathbb{K}(\mathbb{H}_1,\mathbb{H}_2))<1.$ Then for $A\in \mathbb{L}(\mathbb{H}_1,\mathbb{H}_2),~T\perp_B A$ if and only if there exists $x\in M_T$ such that $Tx\perp Ax.$
\end{theorem}
\begin{proof}
If there exists $x\in M_T$ such that $Tx\perp Ax,$ then it is easy to observe that $T\perp_B A.$ Conversely, suppose that $T\perp_B A.$ Then by \cite[Th. 1.1, Page 170]{S}, there exists $\lambda_i\geq 0, f_i\in Ext~J(T)$ for $1\leq i\leq 3$ such that $\lambda_1+\lambda_2+\lambda_3=1$ and $(\lambda_1f_1+\lambda_2f_2+\lambda_3f_3)(A)=0.$ By Lemma \ref{lemma-wojcik}, there exists $x_i\in M_T,y_i^*\in Ext~J(Tx_i)$ for $1\leq i\leq 3$ such that $f_i=y_i^*\otimes x_i.$ Since $Tx_i$ is smooth, it is easy to observe that $y_i^*\otimes x_i(A)=\langle Ax_i,Tx_i\rangle.$ Thus, 
\begin{eqnarray*}
\sum_{i=1}^3 \lambda_if_i(A)&=&0\\
\Rightarrow \sum_{i=1}^3 \lambda_i y_i^*\otimes x_i(A)&=&0
\end{eqnarray*}	
\begin{eqnarray}\label{eq-convex}
\Rightarrow \sum_{i=1}^3 \lambda_i \langle Ax_i,Tx_i\rangle=0.
\end{eqnarray}

Consider the set $ W = \{ \langle Ax, Tx \rangle :x \in S_{H_0}\} = \{ \langle Ax, Tx \rangle : x \in M_T\}$. Since $H_0$ is a subspace of $ \mathbb{H}_1,$ following the idea of \cite[Th. 1]{DMS}, it can be easily verified that the set $ W$  is convex. Now,  it follows from (\ref{eq-convex}) that $0$ is in the convex hull of $W,$ i.e., $0\in W.$ Therefore,  there exists $ x \in M_T$ such that $\langle Ax, Tx \rangle = 0 $ and so $ Tx \perp Ax.$
This completes the proof of the theorem.
\end{proof}

Now, we are ready to prove our desired theorem.

\begin{theorem}
	Let $\mathbb{H}$ be a separable complex Hilbert space and  $T\in S_{\mathbb{L}(\mathbb{H})}.$ Then $T$ is $n^2-$smooth if and only if $M_T=S_{H_0},$ where $H_0$ is a finite-dimensional subspace of $\mathbb{H}$ with $\dim(H_0)=n$ and $\|T\|_{H_0^\perp}<1.$ \\
In case the Hilbert space is real then the result still holds good with $n^2-$smoothness replaced by  ${{n+1}\choose {2}}-$smoothness. 
\end{theorem}
\begin{proof}
   First suppose that $\mathbb{H}$ is complex Hilbert space. The sufficient part of the theorem follows from Theorem \ref{th-hil}. We only prove the necessary part. Suppose that $T$ is $n^2-$smooth. Since $\mathbb{H}$ is separable, by \cite[Th. 3.8]{LR}, $\mathbb{L}(\mathbb{H})/\mathbb{K}(\mathbb{H})$ has no operator whose order of smoothness is finite. Hence, by \cite[Remark 3.7]{LR}, $dist(T,\mathbb{K}(\mathbb{H}))<1.$	Thus, by Lemma \ref{lemma-wojcik}, $M_T\neq \emptyset.$ From \cite[Th. 2.2]{SP}, we get $M_T=S_{H_0}$ for some subspace $H_0$ of $\mathbb{H}.$ Now, let $A\in \mathbb{L}(\mathbb{H})$ be such that $T\perp_BA.$ Then using Theorem \ref{th-complexhil}, we get $x\in M_T$ such that $Tx\perp Ax.$ Thus, by \cite[Th. 3.1]{PSG}, $H_0$ is a finite dimensional subspace of $\mathbb{H}$ and $\|T\|_{H_0^\perp}<1.$ Let $\dim(H_0)=k.$ Then from Theorem \ref{th-hil}, we get $T$ is $k^2-$smooth. Therefore, $k^2=n^2,$ i.e., $k=n.$ Thus, $\dim(H_0)=n.$\\
   The proof of the theorem for real Hilbert space follows similarly using \cite[Th. 3.2]{MPRS}. 
\end{proof}

\section{$k-$smoothness of operators defined on Banach spaces}
In this section, we study $k-$smoothness of operators defined between Banach spaces. We begin with the following simple lemma, the proof of which is given for the sake of completeness.
\begin{lemma}\label{lemma-ind}
	Suppose  $\mathbb{X},\mathbb{Y}$ are  Banach spaces. If $\{x_1,x_2,\ldots,x_m\}$  and $\{y_1^*,y_2^*,$ $\ldots,y_n^*\}$ are linearly independent subsets of $\mathbb{X}$   and  $\mathbb{Y}^*$ respectively, then $\{y_i^*\otimes x_j:1\leq i\leq n,~1\leq j\leq m\}$ is a linearly independent subset of $\mathbb{L}(\mathbb{X},\mathbb{Y})^*.$  	
\end{lemma}
\begin{proof}
Let $c_{ij}$ be scalars such that 
\begin{eqnarray}\label{eq01}
\sum_{1\leq i\leq n,1\leq j\leq m}c_{ij}y_i^*\otimes x_j=0.  
\end{eqnarray}	
Choose $y\in \mathbb{Y},\phi\in \mathbb{X}^*.$ Define $S\in \mathbb{L}(\mathbb{X},\mathbb{Y})$ by $Sx=\phi(x)y$ for all $x\in \mathbb{X}.$ Now, from  (\ref{eq01}) we get, 
\begin{eqnarray*}
&& \sum_{1\leq i\leq n,1\leq j\leq m}c_{ij}y_i^*\otimes x_j(S)=0\\
&\Rightarrow & \sum_{1\leq i\leq n,1\leq j\leq m}c_{ij}y_i^* S(x_j)=0\\
&\Rightarrow&\sum_{1\leq i\leq n,1\leq j\leq m}c_{ij}\phi(x_j)y_i^*(y) =0\\
&\Rightarrow&\phi \Bigg(\sum_{1\leq i\leq n,1\leq j\leq m}c_{ij}x_jy_i^*(y)\Bigg) =0\\
&\Rightarrow& \sum_{1\leq i\leq n,1\leq j\leq m}c_{ij}x_jy_i^*(y) =0,~(\text{since~}\phi\in \mathbb{X}^*~\text{is~arbitrary})\\
&\Rightarrow&\sum_{1\leq i\leq n}c_{ij}y_i^*(y) =0~\text{for~all~}1\leq j\leq m\\
&\Rightarrow&\sum_{1\leq i\leq n}c_{ij}y_i^*=0~(\text{since~}y\in\mathbb{Y}~\text{is~ arbitrary})\\
&\Rightarrow& c_{ij}=0 ~\text{for~all~}1\leq j\leq m,1\leq i\leq n.
\end{eqnarray*}
Thus, $\{y_i^*\otimes x_j:1\leq i\leq n,1\leq j\leq m\}$ is a linearly independent subset of $\mathbb{L}(\mathbb{X},\mathbb{Y})^*.$  	
\end{proof}
We are in a position to prove the following theorem which gives a sufficient condition for $k-$smoothness of operators defined between infinite dimensional Banach spaces, which improves on \cite[Th. 2.2]{MP}.
\begin{theorem}\label{th-ind}
	Suppose $\mathbb{X}$ is a reflexive Banach space and $\mathbb{Y}$ is an arbitrary Banach space. Let $\mathbb{K}(\mathbb{X},\mathbb{Y})$ be an $M-$ideal in $\mathbb{L}(\mathbb{X},\mathbb{Y}).$ Suppose that $T\in S_{\mathbb{L}(\mathbb{X},\mathbb{Y})}$ is such that $dist(T,\mathbb{K}(\mathbb{X},\mathbb{Y}))<1$ and  $M_T\cap Ext(B_{\mathbb{X}})=\{\pm x_1,\pm x_2,\ldots,\pm x_r\},$ where $\{x_1,x_2,\ldots,x_r\}$ is linearly independent in $\mathbb{X}.$ Let $Tx_i$ be $m_i-$smooth for each $1\leq i\leq r.$  Then $T$ is $k-$smooth, where $m_1+m_2+\ldots+m_r=k.$ 
\end{theorem}
\begin{proof}
	Since $Tx_i$ is $m_i-$smooth, we have, $m_i= \dim~ span~ Ext~ J(Tx_i),$ for each $1\leq i\leq r.$ 
	Let $\{ y_{ij}^*\in Ext ~J(Tx_i):1\leq j\leq m_i\}$ be a basis of $span~Ext~J(Tx_i)$ for each $1\leq i\leq r.$ Using similar arguments as in Lemma \ref{lemma-ind}, it can be shown that $\{y_{ij}^*\otimes x_i:1\leq i\leq r,1\leq j\leq m_i \}$ is linearly independent. Now, using Lemma \ref{lemma-wojcik}, we get,
	\begin{eqnarray*}
		&&  span~ Ext ~J(T)\\
		&=&  span ~\{y_{ij}^*\otimes x_i:y_{ij}^*\in Ext ~J(Tx_i),1\leq i\leq r \}\\
		&=& span ~\{y_{ij}^*\otimes x_i:1\leq i\leq r,1\leq j\leq m_i \}.
	\end{eqnarray*}
	Therefore, $\dim~span~Ext~J(T)=m_1+m_2+\ldots+m_r.$ Thus, $T$ is $k-$smooth, where $k=m_1+m_2+\ldots+m_r.$ This completes the proof of the theorem.
\end{proof}

The following corollary now easily follows from Theorem \ref{th-ind}. Once again, we recall that $\mathbb{K}(\ell_p)$ is an $M-$ideal in $\mathbb{L}(\ell_p),$ where $1<p<\infty.$
\begin{corollary}\label{cor-lp}
	Let $T\in S_{\mathbb{L}(\ell_p)},$ where $1<p<\infty.$ Suppose that $dist(T,\mathbb{K}(\ell_p))<1$ and  $M_T=\{\pm x_1,\pm x_2,\ldots,\pm x_k\},$ where $\{x_1,x_2,\ldots,x_k\}$ is linearly independent in $\ell_p.$ Then $T$ is $k-$smooth.
\end{corollary}
\begin{proof}
	For $1<p<\infty,$ $\ell_p$ is strictly convex, smooth space. Hence, $x_i\in Ext(B_{\ell_p})$ and $Tx_i$ is smooth for each $1\leq i\leq k.$ Thus, by Theorem \ref{th-ind}, $T$ is $k-$smooth.
\end{proof}

Now, we exhibit an easy example to show that the converse of Corollary \ref{cor-lp} is not true, i.e., there exists $k-$smooth operator $T\in \mathbb{L}(\ell_p)$ such that $M_T$ is not of the form $\{\pm x_1,\pm x_2,\ldots,\pm x_k\},$ where $\{x_1,x_2,\ldots,x_k\}$ is linearly independent in $\ell_p.$

\begin{example}\label{ex-lp}
	Consider the operator $T\in \mathbb{L}(\ell_p),~(1<p<\infty)$ defined by $T(\sum_{i}a_ie_i)$ $=a_1e_1+a_2e_2,$ where $\{e_i:i\in \mathbb{N}\}$ is the canonical basis of $\ell_p.$ Then $M_T=span\{e_1,e_2\}.$ Let $x=ae_1+be_2,$ where $a,b\neq 0.$ Then $ x \in M_T $ and  so by Lemma \ref{lemma-wojcik}, $e_1^*\otimes e_1,e_2^*\otimes e_2,x^*\otimes x\in Ext~J(T).$ Now, it is easy to show that $\{e_1^*\otimes e_1,e_2^*\otimes e_2,x^*\otimes x\}$ is linearly independent. Therefore, $T$ is $k-$smooth, where $k\geq 3.$ But $M_T$ does not contain $3$ linearly independent vectors. 
\end{example} 
\begin{remark}
 Example \ref{ex-lp} illustrates the fact that one part of the Theorem \cite[Th. 2.3]{KS}, namely $(i) \Rightarrow (ii)$, is not correct and  Theorem \ref{th-ind}  improves on the other part of the same theorem.
\end{remark}

Next, we study the $k-$smoothness  of bounded linear operator $T$ for which $ M_T \cap Ext(B_{\mathbb{X}})$ may contain contain linearly dependent vectors.

\begin{theorem}\label{th-result 7}
	Let $\mathbb{X}$ be a reflexive Banach space and $\mathbb{Y}$ be a finite-dimensional Banach spaces with $\dim(\mathbb{Y})=m.$ Let $T \in S_{\mathbb{L}(\mathbb{X}, \mathbb{Y})}$ be such that $\{x_1,x_2,\ldots, x_r\}  \subseteq M_T \cap Ext(B_{\mathbb{X}}) \subseteq span\{x_1,x_2,\ldots,x_r\},$  where $ \{ x_1, x_2, \ldots, x_r\} $ is linearly independent.  Suppose $Tx_i$ is $m-$smooth for $i=1,2,\ldots,r.$ Then $T$ is $mr-$smooth.
\end{theorem}
\begin{proof}
	For each $1\leq i\leq r,~Tx_i$ is $m-$smooth. Suppose $\{y_{ij}^*:1\leq j\leq m\}$ is a linearly independent subset of $Ext~J(Tx_i)$ for each $1\leq i\leq r.$ We first show that $\dim~span~J(T)\leq mr.$ Since $\mathbb{Y}$ is finite-dimensional, $\mathbb{L}(\mathbb{X},\mathbb{Y})=\mathbb{K}(\mathbb{X},\mathbb{Y}).$ Hence, $T$ is compact operator and $\mathbb{K}(\mathbb{X},\mathbb{Y})$ is trivially an $M-$ideal in $\mathbb{L}(\mathbb{X},\mathbb{Y}).$ Thus, by Lemma \ref{lemma-wojcik},
	\[Ext ~J(T)=\{y^*\otimes x\in \mathbb{K}(\mathbb{X},\mathbb{Y})^*:x\in M_T\cap Ext(B_{\mathbb{X}}), y^*\in Ext ~J(Tx)\}.\]
	  Let $x \in M_T \cap Ext(B_{\mathbb{X}}).$ Then there exist scalars $c_i,i=1,2,...,r$ such that $x=c_1x_1+c_2x_2+...+c_rx_r.$ Now, let $y^* \in Ext~J(Tx).$ Since $\{y_{1j}^*:1\leq j\leq m \}$ is linearly independent, and hence forms a basis of $\mathbb{Y}^*,$ there exist scalars $d_{j}~( 1 \leq j \leq m)$ such that $y^*=\sum_{1\leq j\leq m}d_jy_{1j}^*.$ Thus, 
	\begin{eqnarray*}
		y^* \otimes x 
		&=& y^* \otimes (c_1x_1+c_2x_2+...+c_rx_r)\\
		&=& \Big(\sum_{1\leq j\leq m}d_jy_{1j}^*\Big)\otimes \Big(\sum_{1\leq i\leq r}c_ix_i\Big)\\
		&=& \sum_{1\leq i\leq r,1\leq j\leq m}c_id_{j}y_{1j}^*\otimes x_i\\
		&\in & span \{y_{1j}^*\otimes x_i:1\leq i\leq r,1\leq j\leq m\}.
	\end{eqnarray*}
	Since $x \in M_T \cap Ext(B_{\mathbb{X}})$ and $y^* \in Ext~J(Tx)$ are arbitrary, we have $	Ext~J(T) \subseteq  span \{y_{1j}^*\otimes x_i:1\leq i\leq r,1\leq j\leq m\}.$ Now,
	\begin{eqnarray*}
		&& \dim~span~ Ext ~J(T)\\
		&\leq& \dim~span~\{y_{1j}^*\otimes x_i:1\leq i\leq r,1\leq j\leq m\}\\
		&=& mr,~(\text{by~Lemma ~\ref{lemma-ind}}).
	\end{eqnarray*}
	
	Therefore, $T$ is $k-$smooth, where $k\leq mr.$ Now, $y_{ij}^*\otimes x_i\in Ext~J(T)$ for all $1\leq i\leq r,1\leq j\leq m.$ Using similar arguments as in Lemma \ref{lemma-ind}, it can be shown that $\{y_{ij}^*\otimes x_i:1\leq i\leq r,1\leq j\leq m\}$ is linearly independent subset of $Ext~J(T).$ Thus, $\dim~span~Ext~J(T)\geq mr,$ i.e., $k\geq mr.$ Therefore, $k=mr$ and $T$ is $mr-$smooth. This completes the proof of the theorem.
\end{proof}
To illustrate the usefulness of Theorem \ref{th-result 7} we cite the following example for which $k-$smoothness of the operator can not be obtained using Theorem \ref{th-ind} or \cite[Th. 2.2]{MP} but can be obtained using above theorem.

\begin{example}
	Let $\mathbb{X}=\ell_{\infty}^4$ and $\mathbb{Y}$ be a two-dimensional Banach space such that $B_\mathbb{Y}$ is the convex hull of $\{\pm (2,1),\pm (2,-1)\}.$ Define $T\in \mathbb{L}(\mathbb{X},\mathbb{Y})$ by $T(x,y,z,w)=(y+w,x).$ Then $$M_T\cap Ext(B_\mathbb{X})=\{\pm(1,1,1,1),\pm(1,1,-1,1),\pm(-1,1,1,1),\pm(1,-1,1,-1)\}.$$ Clearly, $\{(1,1,1,1),(1,1,-1,1),(-1,1,1,1),(1,-1,1,-1)\}$ is linearly dependent. Therefore, the order of smoothness of $T$ cannot be obtained from Theorem \ref{th-ind} or \cite[Th. 2.2]{MP}. Observe that,  $M_T\cap Ext(B_\mathbb{X})\subseteq ~span~\{\pm(1,1,-1,1),\pm(-1,1,1,1),$ $ \pm(1,-1,1,-1)\},$ where $\{(1,1,-1,1),(-1,1,1,1),(1,-1,1,-1)\}$ is linearly independent. Moreover, $T(1,1,-1,1),$ $T(-1,1,1,1),T(1,-1,1,-1)$  are $2-$smooth. Therefore, using Theorem \ref{th-result 7} we get, $T$ is $6-$smooth.
\end{example}

We now turn our attention to the study of $k-$smoothness of  operators in the setting of  special Banach spaces.  We characterize the $k-$smoothness of an operator defined from a finite-dimensional Banach space to $\ell_{\infty}^n,~(n\in \mathbb{N}).$ To do so we need  \cite[Cor.2.3]{MP} and the following proposition which gives a nice relation between the order of smoothness of an operator and its adjoint.

\begin{proposition}\label{prop-1}
	Let $\mathbb{X},\mathbb{Y}$ be finite-dimensional Banach spaces. Let $T\in S_{\mathbb{L}(\mathbb{X},\mathbb{Y})}.$ Then $T$ is $k-$smooth if and only if $T^*$ is $k-$smooth.
\end{proposition}
\begin{proof}
	We first show that $Ext~J(T)=Ext~J(T^*).$ Let $y^*\otimes x\in Ext~J(T).$ Then $x\in M_T\cap Ext(B_\mathbb{X})$ and $y^*\in Ext~J(Tx).$ Now, 
	$$y^*\otimes x(T)=1\Rightarrow y^*(Tx)=1\Rightarrow x(T^*y^*)=1\Rightarrow \|T^*y^*\|=1.$$
	Thus, $y^*\in M_{T^*}$ and $x\in J(T^*y^*).$ Moreover, $x\in  Ext(B_\mathbb{X})$ and $Ext~J(Tx)\subseteq Ext(B_{\mathbb{Y}^*}).$ Therefore, $x\in Ext~J(T^*y^*)$ and $y^*\in M_{T^*}\cap Ext(B_{\mathbb{Y}^*})$ and so $y^*\otimes x\in Ext~J(T^*).$ Hence, $Ext~J(T)\subseteq Ext~J(T^*).$ Now, replacing $T$ by $T^*,$ we get $Ext~J(T^*)\subseteq Ext~J(T).$ Thus, $ExtJ(T)=Ext~J(T^*),$ i.e., $\dim~span~Ext~J(T)=\dim~span~Ext~J(T^*).$ Therefore, $T$ is $k-$smooth if and only if $T^*$ is $k-$smooth. 
\end{proof}

\begin{corollary}
	Let $\mathbb{X}$ be a finite-dimensional Banach space. Let $T\in S_{\mathbb{L}(\mathbb{X},\ell_{\infty}^n)}.$ Then $T$ is $k-$smooth if and only if $M_{T^*}\cap Ext(B_{\ell_{1}^n})=\{\pm e_1,\pm e_2,\ldots \pm e_r\}$ for some $1\leq r\leq n,~T^*e_i$ is $m_i-$smooth for each $1\leq i\leq r$ and $m_1+m_2+\ldots +m_r=k.$ 
\end{corollary}
\begin{proof}
	From Proposition \ref{prop-1}, $T$ is $k-$smooth if and only if $T^*$ is $k-$smooth. Now, $T^*\in S_{\mathbb{L}(\ell_1^n,\mathbb{X}^*)}.$ Moreover, from \cite[Cor. 2.3]{MP}, we get that $T^*$ is $k-$smooth if and only if $M_{T^*}\cap Ext(B_{\ell_{1}^n})=\{\pm e_1,\pm e_2,\ldots \pm e_r\}$ for some $1\leq r\leq n,~T^*e_i$ is $m_i-$smooth for each $1\leq i\leq r$ and $m_1+m_2+\ldots +m_r=k.$ This completes the proof of the corollary.	
\end{proof}

 We  next determine the order of smoothness of $T\in \mathbb{L}(\ell_{\infty}^3,\mathbb{Y}),$ where $\mathbb{Y}$ is an arbitrary two-dimensional Banach space, depending on $|M_T\cap Ext(B_{\ell_{\infty}^3})|.$ Observe that if $|M_T\cap Ext(B_{\ell_{\infty}^3})|\leq 6,$ then the order of smoothness of $T$ can be obtained using Theorem \ref{th-ind}. Therefore, we only consider the case for which $|M_T\cap Ext(B_{\ell_{\infty}^3})|=8,$ i.e., $M_T\cap Ext(B_{\ell_{\infty}^3})=\{\pm x_1,\pm x_2,\pm x_3,\pm x_4\},$ where $x_1=(1,1,1),x_2=(-1,1,1),x_3=(-1,-1,1),x_4=(1,-1,1).$ Note that, for each $1\leq i\leq 4,~Tx_i$ is either smooth or $2-$smooth.

\begin{theorem}\label{th-001}
	Let $\mathbb{X}=\ell_{\infty}^3$ and $\mathbb{Y}$ be two-dimensional Banach space. Let $T \in S_{\mathbb{L}(\mathbb{X}, \mathbb{Y})}$ be such that $M_T \cap Ext(B_{\mathbb{X}})= \{ \pm x_1, \pm x_2, \pm x_3, \pm x_4 \}.$ Suppose $S_1=\{x_i:1\leq i\leq 4,~Tx_i~\text{is~}\text{smooth}\}.$ Then the following hold:\\
	\textbf{(I)} Let $|S_1|=4$.\\
	$(a)$ If either $Rank(T)=1$ or for each $i, j (1 \leq i \neq j \leq 4)$ either $Tx_i, Tx_j$ or $Tx_i, -Tx_j$ belong to the same straight line  contained in $S_{\mathbb{Y}},$ then $T$ is $3-$smooth.\\
	$(b)$ Otherwise, $T$ is $4-$smooth.\\
	\textbf{(II)} If $|S_1|=3,$ then  $T$ is $4-$smooth.\\
	\textbf{ (III) } If $|S_1|=2,$ then  $T$ is $5-$smooth.\\
	\textbf{ (IV)} If $|S_1|<2,$ then $S_1=\emptyset$ and $T$ is $6-$smooth.
\end{theorem}

\begin{proof}
	Clearly, $T$ is $k-$smooth for $1 \leq k \leq 6,$ since $\dim(\mathbb{X})=3$ and $\dim(\mathbb{Y})=2.$
	
	\textbf{(I) } Let $|S_1|=4.$ Then $Tx_i$ is smooth for $1 \leq i \leq 4$.
	
	$(a)$    If the given condition is satisfied, then it is clear that there exists $y^* \in S_{\mathbb{Y}^*}$ such that for all $i (1 \leq i \leq 4),$ $J(Tx_i)=\{y^* \}$ or $\{-y^* \}.$ Now, if $T$ is $k-$smooth, then
	\begin{eqnarray*}
		k&=&\dim ~span ~J(T)\\
		&=& \dim~ span~ Ext ~J(T)\\
		&=& \dim~ span ~\{y^*\otimes x_1, y^*\otimes x_2, y^*\otimes x_3, y^*\otimes x_4 \}\\
		&=& \dim~ span ~\{y^*\otimes x_1, y^*\otimes x_2, y^*\otimes x_3 \}\\
		&=& 3,~(\text{by Lemma~\ref{lemma-ind}}).
	\end{eqnarray*}
	Hence, $T$ is $3-$smooth.
	
	$(b)$    Suppose the condition $(a)$ is not satisfied. Thus, there exist $1\leq i, j\leq 4$ such that $Tx_i \neq \pm Tx_j$ and neither $Tx_i, Tx_j$ nor $Tx_i, -Tx_j$ belong to the same straight line contained in $S_{\mathbb{Y}}.$ Without loss of generality, we assume $i=1,~ j=2.$ Let $J(Tx_i)=\{y_i^* \}, (1 \leq i \leq 4).$ Then it is easy to observe that $\{y_1^*, y_2^* \}$ is linearly independent. Let $y_3^*=ay_1^*+by_2^*$ and $y_4^*=cy_1^*+dy_2^*,$ where $a,b,c,d\in \mathbb{R}.$ Since $\|y_3^*\|=1,$ $a$ and $b$ cannot be zero simultaneously. Similarly, $c$ and $d$ cannot be zero simultaneously.  Now, if $T$ is $k-$smooth, then
	\begin{eqnarray*}
		k&=&\dim ~span ~J(T)\\
		&=& \dim~ span~ Ext ~J(T)\\
		&=& \dim~ span ~\{y_1^*\otimes x_1, y_2^*\otimes x_2, y_3^*\otimes x_3, y_4^*\otimes x_4 \}.
	\end{eqnarray*}
	We show that $\{y_1^*\otimes x_1, y_2^*\otimes x_2, y_3^*\otimes x_3, y_4^*\otimes x_4 \}$ is linearly independent. Let $c_i(1\leq i\leq 4)\in \mathbb{R}$ be such that \\
	$$c_1y_1^*\otimes x_1+c_2 y_2^*\otimes x_2+c_3 y_3^*\otimes x_3+c_4 y_4^*\otimes x_4 =0.$$ Then \\
	$c_1y_1^*\otimes x_1+c_2 y_2^*\otimes x_2+c_3 (ay_1^*+by_2^*)\otimes x_3+c_4 (cy_1^*+dy_2^*)\otimes (x_1-x_2+x_3) =0.$\\
	$\Rightarrow (c_1+c_4c)y_1^*\otimes x_1+(c_2-c_4d)y_2^*\otimes x_2+(c_3a+c_4c)y_1^*\otimes x_3+(c_3b+c_4d)y_2^*\otimes x_3-c_4cy_1^*\otimes x_2+c_4dy_2^*\otimes x_1=0.$\\
	Since $\{x_1,x_2,x_3\}$ is linearly independent subset of $\mathbb{X}$ and $\{y_1^*,y_2^*\}$ is linearly independent subset of $\mathbb{Y}^*,$ from Lemma  \ref{lemma-ind}, we get that $\{y_1^*\otimes x_1,y_2^*\otimes x_1,y_1^*\otimes x_2,y_2^*\otimes x_2,y_1^*\otimes x_3,y_2^*\otimes x_3\}$ is linearly independent subset of $\mathbb{L}(\mathbb{X},\mathbb{Y})^*$. Therefore, $$c_1+c_4c=0,~c_2-c_4d=0,~c_3a+c_4c=0,~c_3b+c_4d=0,~c_4c=0,~c_4d=0.$$
	Now, solving these equations, we obtain $c_i=0$ for all $1\leq i\leq 4.$ Hence, $\{y_1^*\otimes x_1, y_2^*\otimes x_2, y_3^*\otimes x_3, y_4^*\otimes x_4 \}$ is linearly independent. Thus, $k=4$ and $T$ is $4-$smooth. \\

	\textbf{(II)}  Without loss of generality, assume that $S_1=\{x_2,x_3,x_4\},$ i.e., $Tx_2,Tx_3,Tx_4$ are smooth points of $S_\mathbb{Y}$ and $Tx_1$ is $2-$smooth point of $S_\mathbb{Y}.$ Clearly, $Rank(T)=2.$ Now, by \cite[Lemma 2.11]{MPD}, $T(B_\mathbb{X})$ is a polygon with $4$ extreme points. Since $Tx_1$ is a $2-$smooth point of $S_{\mathbb{Y}},$ by \cite[Th. 4.1]{KS}, $\pm Tx_1$ must be  two extreme points of $B_\mathbb{Y}.$ Since $\|T\|=1, ~T(B_\mathbb{X})\subseteq B_\mathbb{Y},$ i.e., $Ext(B_\mathbb{Y})\cap T(B_\mathbb{X})\subseteq Ext(T(B_\mathbb{X})).$ Therefore, $\pm Tx_1\in Ext(T(B_\mathbb{X})).$ Suppose that the other two extreme points of the polygon $T(B_\mathbb{X})$ are $\pm Tx_2.$ Then $L[Tx_2,-Tx_1]$ is an edge  of the polygon $T(B_\mathbb{X}).$ Now, $L[x_2,-x_1]\cap L[x_3,-x_4]=\{(-1,0,0)\}.$ Therefore, $T(-1,0,0)\in L[Tx_2,-Tx_1]\cap L[Tx_3,-Tx_4].$ This implies that $Tx_3,-Tx_4\in L[Tx_2,-Tx_1].$ Since $Tx_3,Tx_4$ are smooth points and $Tx_1$ is $2-$smooth point, $Tx_3\neq -Tx_1$ and $-Tx_4\neq -Tx_1.$ Now, $Rank(T)=2$ implies that either $Tx_3\neq Tx_2$ or $-Tx_4\neq Tx_2.$ Without loss of generality, let us assume that $Tx_3\neq Tx_2.$ Then $Tx_3\in L(Tx_2,-Tx_1).$ Since $\|Tx_3\|=1,$  $L[Tx_2,-Tx_1]\subseteq S_\mathbb{Y}.$ Now, let $J(Tx_3)=\{y^*\}.$ Then it is easy to see that $J(Tx_2)=\{y^*\}$ and $J(Tx_4)=\{-y^*\}.$ Since $Tx_1$ is $2-$smooth, it is clear that  $Ext~J(Tx_1)=\{y_1^*, y_2^* \},$ where $\{y_1^*,y_2^*\}$ is a linearly independent subset of $\mathbb{Y}^*.$  Let $y^*=ay_1^*+by_2^*.$ Now, if $T$ is $k-$smooth, then 
	\begin{eqnarray*}
		k&=&\dim ~span ~J(T)\\
		&=& \dim~ span~ Ext ~J(T)\\
		&=& \dim~ span ~\{y_1^*\otimes x_1, y_2^*\otimes x_1, y^*\otimes x_2, y^*\otimes x_3, y^*\otimes x_4\}\\
		&=& \dim~ span ~\{y_1^*\otimes x_1, y_2^*\otimes x_1, y^*\otimes x_2, y^*\otimes x_3 \}.
	\end{eqnarray*}
	Now, using Lemma \ref{lemma-ind} it can be observed that $\{y_1^*\otimes x_1, y_2^*\otimes x_1, y^*\otimes x_2, y^*\otimes x_3 \}$ is a linearly independent subset of $\mathbb{L}(\mathbb{X},\mathbb{Y})^*.$  Hence, $k=4$ and $T$ is $4-$smooth.\\
	
	\textbf{(III) }  Without loss of generality, we assume $S_1=\{x_3,x_4\}.$ Then $\pm Tx_3, \pm Tx_4$ are smooth points of $S_\mathbb{Y}$ and $\pm Tx_1, \pm Tx_2$ are $2-$smooth points of $S_\mathbb{Y}.$ Clearly, $Rank(T)=2.$ By \cite[Th. 4.1]{KS}, $\pm Tx_1,\pm Tx_2\in Ext(B_\mathbb{Y}).$ Since $\|T\|=1,T(B_\mathbb{X})\subseteq B_\mathbb{Y}.$ This gives that $Ext(B_\mathbb{Y})\cap T(B_\mathbb{X})\subseteq Ext(T(B_\mathbb{X})).$ Thus, we get, $\pm Tx_1,\pm Tx_2\in Ext(T(B_\mathbb{X})).$ If possible, suppose that $Tx_1=-Tx_2.$ Then $x_4=x_1-x_2+x_3$ implies that $Tx_1=\frac{Tx_4-Tx_3}{2}.$ Since $Tx_1\in Ext(B_\mathbb{Y}),$ we must have $Tx_1=Tx_4=-Tx_3.$ Thus, $Rank(T)=1,$ a contradiction. Therefore, $Tx_1\neq -Tx_2.$ First assume that $Tx_1=Tx_2.$ Then from $x_4=x_1-x_2+x_3,$ we get $Tx_3=Tx_4.$ Let $J(Tx_3)=J(Tx_4)=\{y^*\}$ and $Ext~J(Tx_1)=Ext~J(Tx_2)=\{y_1^*,y_2^*\}.$ Then it is easy to see that $T$ is $5-$smooth. Now, assume that $Tx_1\neq Tx_2.$ Then  $\pm Tx_1,\pm Tx_2$ are $4$ distinct extreme points of the polygon $T(B_\mathbb{X})$ and $L[Tx_2,-Tx_1]$ is an edge of $T(B_\mathbb{X}).$  Now, as in \textbf{(II)}, it can be shown that $Tx_3,-Tx_4\in L[Tx_2,-Tx_1].$ Since $Tx_3,-Tx_4$ are smooth points, $Tx_3,-Tx_4\in L(Tx_2,-Tx_1).$ From $\|Tx_3\|=1,$ we can show that $L[Tx_2,-Tx_1]\subseteq S_\mathbb{Y}.$  Let $J(Tx_4)=\{y^*\}.$ Then $J(Tx_3)=\{-y^*\}.$ Let $Ext~J(Tx_1)=\{y_1^*, y_2^* \}$ and $Ext~J(Tx_2)=\{y_3^*, y_4^* \}.$ Clearly, $\{y_1^*, y_2^* \}$ and $\{y_3^*, y_4^* \}$ are linearly independent. Now, if $T$ is $k-$smooth, then
	\begin{eqnarray*}
		k&=&\dim ~span ~J(T)\\
		&=& \dim~ span~ Ext ~J(T)\\
		&=& \dim~ span ~\{y_1^*\otimes x_1, y_2^*\otimes x_1, y_3^*\otimes x_2, y_4^*\otimes x_2, y^*\otimes x_3, y^*\otimes x_4\}\\
		&=& \dim~ span ~\{y_1^*\otimes x_1, y_2^*\otimes x_1, y_3^*\otimes x_2, y_4^*\otimes x_2, y^*\otimes x_3 \}\\
		&=& \dim~ span ~\{y_1^*\otimes x_1, y_2^*\otimes x_1, y_1^*\otimes x_2, y_2^*\otimes x_2, y^*\otimes x_3 \}\\
		&=& 5, \text{by simple calculation}.
	\end{eqnarray*}
	Hence, $T$ is $5-$smooth.\\
	
	\textbf{ (IV) } Since $|S_1|<2,$ at least $3$ points of $Tx_i,~1\leq i\leq 4$ are $2-$smooth. Without loss of generality, suppose that $Tx_1,Tx_2,Tx_3$ are $2-$smooth. If $Rank(T)=1,$ then it is easy to see that $Tx_4$ is $2-$smooth. Suppose $Rank(T)=2.$ Then by \cite[Lemma 2.11]{MPD}, $T(B_\mathbb{X})$ is a polygon with $4$ extreme points. First let $Ext~(T(B_\mathbb{X}))=\{\pm Tx_1,\pm Tx_2\}.$ Then using similar arguments as in (\textbf{III}) we can show that $Tx_3,$ $-Tx_4 \in L[Tx_2,-Tx_1].$ Since $Tx_3$ is $2-$smooth, we must have either $Tx_3=Tx_2$ or $Tx_3=-Tx_1.$ If $Tx_3=Tx_2,$ then from $x_4=x_1-x_2+x_3,$ we get $Tx_4=Tx_1.$ If $Tx_3=-Tx_1,$ then similarly, we get $Tx_4=-Tx_2.$ In each case, $Tx_4$ is $2-$smooth. Similarly, considering other cases, we can conclude that if $|S_1|<2,$ then $Tx_i$ are $2-$smooth for all $1\leq i\leq 4,$ i.e., $S_1=\emptyset.$ Using Theorem \ref{th-result 7}, we can now say that $T$ is $6-$smooth. This completes the proof of the theorem.
\end{proof}

In Theorem \ref{th-001}, if we further assume that $\mathbb{Y}$ is a two-dimensional strictly convex, smooth Banach space, then we obtain the following corollary.

\begin{corollary}\label{cor-sc}
	Let $\mathbb{X}=\ell_{\infty}^3$ and $\mathbb{Y}$ be a two-dimensional strictly convex, smooth Banach space. Let $T \in S_{\mathbb{L}(\mathbb{X}, \mathbb{Y})}$ and $M_T \cap Ext(B_{\mathbb{X}})= \{ \pm x_1, \pm x_2, \pm x_3, \pm x_4 \}.$ Then the following hold:\\
	$(i)$ If $Rank(T)=1$ then $T$ is $3-$smooth.\\
	$(ii)$ If $Rank(T)=2,$ then $T$ is $4-$smooth.
\end{corollary}
\begin{proof}
	Observe that, since $\mathbb{Y}$ is strictly convex, $S_\mathbb{Y}$ does not contain non-trivial straight line segment. Now, since $\mathbb{Y}$ is smooth, $Tx_i$ is smooth for all $1\leq i\leq 4.$ Thus, the corollary follows from case \textbf{(I)} of Theorem \ref{th-001}. 
\end{proof}

As an immediate application of Theorem \ref{th-001}, we can characterize the extreme contractions defined from $\ell_{\infty}^3$ to arbitrary two-dimensional polygonal Banach space.

\begin{theorem}
	Let $\mathbb{X}=\ell_{\infty}^3$ and $\mathbb{Y}$ be a two-dimensional polygonal Banach space. Let $T\in S_{\mathbb{L}(\mathbb{X},\mathbb{Y})}.$ Then $T$ is an extreme contraction if and only if $|M_T\cap Ext(B_\mathbb{X})|\geq 6$ and $T( M_T\cap Ext(B_\mathbb{X}))\subseteq Ext(B_\mathbb{Y}).$
\end{theorem}
\begin{proof}
	First let $T$ be an extreme contraction. Then by \cite[Th. 2.2]{MPD}, $T$ is $6-$smooth. From \cite[Th. 2.2]{SRP},  we get $span(M_T\cap Ext(B_\mathbb{X}))=\mathbb{X},$ i.e., $|M_T\cap Ext(B_\mathbb{X})|\geq 6.$ Let $|M_T\cap Ext(B_\mathbb{X})|= 6.$ Then $M_T\cap Ext(B_\mathbb{X})$ is of the form $\{\pm x_1,\pm x_2,\pm x_3\},$ where $\{x_1,x_2,x_3\}$ is linearly independent. Now, from Theorem \ref{th-ind} it is clear that $Tx_i$ is $2-$smooth for each $1\leq i\leq 3.$ Therefore, by \cite[Th. 4.1]{KS}, $Tx_i\in Ext(B_\mathbb{Y})$ for all $1\leq i\leq 3.$ Now, suppose that $|M_T\cap Ext(B_\mathbb{X})|= 8.$ Then from Theorem \ref{th-001}, we can conclude that for all $x\in M_T\cap Ext(B_\mathbb{X}),$ $Tx$ is $2-$smooth, i.e., $Tx\in Ext(B_\mathbb{Y}).$ \\
	Conversely, suppose that $|M_T\cap Ext(B_\mathbb{X})|\geq 6$ and  $T( M_T\cap Ext(B_\mathbb{X}))\subseteq Ext(B_\mathbb{Y}).$ If $|M_T\cap Ext(B_\mathbb{X})|= 6,$ then from Theorem \ref{th-ind}, we get, $T$ is $6-$smooth. Hence, by \cite[Th. 2.2]{MPD}, $T$ is an extreme contraction. If $|M_T\cap Ext(B_\mathbb{X})|=8,$ then from Theorem \ref{th-001}, we get $T$ is $6-$smooth. Thus, again by \cite[Th. 2.2]{MPD}, $T$ is an extreme contraction. This completes the proof of the theorem.
\end{proof}

We end this article with the following question:
\begin{question}
Suppose $\mathbb{X} $ and $\mathbb{Y}$ are Banach spaces and $ T \in S_{L(\mathbb{X}, \mathbb{Y})} ,$ then what are the necessary and sufficient conditions for $T$ to be multi-smooth point of finite order?
One can consider the case $\mathbb{X} = \ell_{\infty}^n, \mathbb{Y} = \ell_1^{n}, (n \geq 3).$ There are many more cases where the question is still unanswered.
\end{question}

\bibliographystyle{amsplain}

\end{document}